\documentclass[11pt, reqno]{amsart}

\newcommand{\chapter}{\part}
\usepackage[usenames,dvipsnames]{xcolor}

\usepackage[OT2, T1]{fontenc}
\usepackage{url}
\usepackage{amsmath}
\usepackage{array}
\usepackage{graphicx}
\usepackage{amssymb}
\usepackage{amsthm}
\definecolor{link}{RGB}{11,0,128}
\usepackage[colorlinks=true,citecolor=NavyBlue,linkcolor=OliveGreen,urlcolor=link]{hyperref}
\usepackage{hyperref}
\usepackage{paralist}
\usepackage{colonequals}			
\usepackage{color}
\usepackage{mathabx}			
\usepackage{amscd}
\usepackage[all,cmtip]{xy}			
\usepackage{sseq}				
\usepackage{verbatim}
\usepackage{parskip}
\usepackage{microtype}
\usepackage[in]{fullpage}
\usepackage{mathrsfs} 			
\usepackage{etoolbox}
\usepackage{enumitem}
\usepackage{moreenum}			
\usepackage[alphabetic,lite]{amsrefs} 	
\usepackage{stmaryrd}			
\usepackage[foot]{amsaddr}
\usepackage{subfiles}
\usepackage{ragged2e}
\usepackage{tikz}
\usetikzlibrary{graphs}
\usetikzlibrary{cd}
\usepackage{stackengine}
\usepackage{lipsum}
\usepackage{cleveref}

\calclayout
\newcommand{\cE}{\mathcal{E}}
\newcommand{\cF}{\mathcal{F}}

\newcommand{\cL}{\mathcal{L}}
\newcommand{\cM}{\mathcal{M}}

\newcommand{\cO}{\mathcal{O}}

\newcommand{\cS}{\mathcal{S}}
\newcommand{\cT}{\mathcal{T}}

\newcommand{\cX}{\mathcal{X}}
\newcommand{\cY}{\mathcal{Y}}

\usepackage{relsize}
\usepackage[bbgreekl]{mathbbol}
\usepackage{amsfonts}

\DeclareSymbolFontAlphabet{\mathbb}{AMSb}
\DeclareSymbolFontAlphabet{\mathbbl}{bbold}

\DeclareMathOperator{\Ani}{Ani}		
\DeclareMathOperator{\Jac}{Jac}		

\DeclareSymbolFont{cyrletters}{OT2}{wncyr}{m}{n}
\DeclareMathSymbol{\Sha}{\mathalpha}{cyrletters}{"58}	
\DeclareMathOperator{\Spa}{Spa}		
\DeclareMathOperator{\Spec}{Spec}

\newcommand{\et}{\mathrm{\acute{e}t}}	

\newcommand{\surjects}{\twoheadrightarrow}
\newcommand{\tensor}{\otimes} 			

\providecommand{\up}[1]{{\upshape(}#1{\upshape)}}

\renewcommand{\b}{\textbf}

\newcommand{\brems}{\begin{rems} \hfill \begin{enumerate}[label=\b{\thenumberingbase.},ref=\thenumberingbase]}

\newcommand{\erems}{\end{enumerate} \end{rems}}
\newcommand{\begs}{\begin{egs} \hfill \begin{enumerate}[label=\b{\thenumberingbase.},ref=\thenumberingbase]}

\newcommand{\eegs}{\end{enumerate} \end{egs}}

\newcommand{\bsm}{\begin{smallmatrix}}
\newcommand{\esm}{\end{smallmatrix}}
\newcommand{\blem}{\begin{lemma}}
\newcommand{\elem}{\end{lemma}}

\newcommand{\bconj}{\begin{conj}}
\newcommand{\econj}{\end{conj}}
\newcommand{\bprob}{\begin{Problem}}
\newcommand{\eprob}{\end{Problem}}

\newcommand{\bq}{\begin{Q}}
\newcommand{\eq}{\end{Q}}
\newcommand{\benum}{\begin{enumerate}[label={{\upshape(\alph*)}}]}
\newcommand{\benuma}{\begin{enumerate}[label={{\upshape(\arabic*)}}]}
\newcommand{\benumb}{\begin{enumerate}[label={{\upshape\b{\arabic*.}}}]}
\newcommand{\benumr}{\begin{enumerate}[label={{\upshape(\roman*)}}]}
\newcommand{\eenum}{\end{enumerate}}
\newcommand{\bitem}{\begin{itemize}}
\newcommand{\eitem}{\end{itemize}}
\newcommand{\bc}{}
\newcommand{\bd}{\begin{defn}}
\newcommand{\ed}{\end{defn}}

\newcommand{\beg}{\begin{eg}}
\newcommand{\eeg}{\end{eg}}

\newcommand{\bcl}{\begin{claim}}
\newcommand{\ecl}{\end{claim}}

\newcommand{\ba}{\begin{aligned}}
\newcommand{\ea}{\end{aligned}}
\newcommand{\be}{\begin{equation}}
\newcommand{\ee}{\end{equation}}
\newcommand{\bpf}{\begin{proof}}
\newcommand{\epf}{\end{proof}}
\newcommand{\bthm}{\begin{thm}}
\newcommand{\ethm}{\end{thm}}
\newcommand{\bprop}{\begin{prop}}
\newcommand{\eprop}{\end{prop}}
\newcommand{\bcor}{\begin{cor}}
\newcommand{\ecor}{\end{cor}}
\newcommand{\brem}{\begin{rem}}
\newcommand{\erem}{\end{rem}}

\newcounter{numberingbase}[subsection]

\makeatletter
\newcommand{\nbprefix}{%
  \ifnum\value{subsection}>0 \thesubsection \else \thesection \fi
}
\makeatother
\renewcommand{\thenumberingbase}{\nbprefix.\arabic{numberingbase}}

\numberwithin{equation}{section}

\crefname{numberingbase}{Item}{Items}
\Crefname{numberingbase}{Item}{Items}

\newtheoremstyle{thms}{0.5em}{0.5em}{\itshape}{}{\bfseries}{.}{ }{}
\newtheoremstyle{claims}{0.5em}{0.5em}{}{}{\itshape}{.}{ }{}
\newtheoremstyle{defs}{0.5em}{0.5em}{}{}{\bfseries}{.}{ }{}

\theoremstyle{thms}

\newtheorem{thmX}{Theorem}[subsection]
\renewcommand{\thethmX}{\thenumberingbase}
\newenvironment{thm}{%
  \refstepcounter{numberingbase}\setcounter{thmX}{\value{numberingbase}}%
  \begin{thmX}}{\end{thmX}}
\crefname{thmX}{Theorem}{Theorems}
\Crefname{thmX}{Theorem}{Theorems}

\newtheorem{theoremX}{Theorem}[subsection]

\crefname{theoremX}{Theorem}{Theorems}
\Crefname{theoremX}{Theorem}{Theorems}

\newtheorem{TheoremX}{Theorem}[subsection]

\crefname{TheoremX}{Theorem}{Theorems}
\Crefname{TheoremX}{Theorem}{Theorems}

\newtheorem{conjX}{Conjecture}[subsection]

\newenvironment{conj}{%
  \refstepcounter{numberingbase}\setcounter{conjX}{\value{numberingbase}}%
  \begin{conjX}}{\end{conjX}}
\crefname{conjX}{Conjecture}{Conjectures}
\Crefname{conjX}{Conjecture}{Conjectures}

\newtheorem{corollaryX}{Corollary}[subsection]

\crefname{corollaryX}{Corollary}{Corollaries}
\Crefname{corollaryX}{Corollary}{Corollaries}

\newtheorem{corX}{Corollary}[subsection]

\newenvironment{cor}{%
  \refstepcounter{numberingbase}\setcounter{corX}{\value{numberingbase}}%
  \begin{corX}}{\end{corX}}
\crefname{corX}{Corollary}{Corollaries}
\Crefname{corX}{Corollary}{Corollaries}

\newtheorem{CorollaryX}{Corollary}[subsection]

\crefname{CorollaryX}{Corollary}{Corollaries}
\Crefname{CorollaryX}{Corollary}{Corollaries}

\newtheorem{lemmaX}{Lemma}[subsection]

\newenvironment{lemma}{%
  \refstepcounter{numberingbase}\setcounter{lemmaX}{\value{numberingbase}}%
  \begin{lemmaX}}{\end{lemmaX}}
\crefname{lemmaX}{Lemma}{Lemmas}
\Crefname{lemmaX}{Lemma}{Lemmas}

\newtheorem{LemmaX}{Lemma}[subsection]

\crefname{LemmaX}{Lemma}{Lemmas}
\Crefname{LemmaX}{Lemma}{Lemmas}

\newtheorem{lemX}{Lemma}[subsection]

\crefname{lemX}{Lemma}{Lemmas}
\Crefname{lemX}{Lemma}{Lemmas}

\newtheorem{propX}{Proposition}[subsection]
\renewcommand{\thepropX}{\thenumberingbase}
\newenvironment{prop}{%
  \refstepcounter{numberingbase}\setcounter{propX}{\value{numberingbase}}%
  \begin{propX}}{\end{propX}}
\crefname{propX}{Proposition}{Propositions}
\Crefname{propX}{Proposition}{Propositions}

\newtheorem{PropositionX}{Proposition}[subsection]

\crefname{PropositionX}{Proposition}{Propositions}
\Crefname{PropositionX}{Proposition}{Propositions}

\newtheorem{propositionX}{Proposition}[subsection]

\crefname{propositionX}{Proposition}{Propositions}
\Crefname{propositionX}{Proposition}{Propositions}

\newtheorem{QX}{Question}[subsection]

\newenvironment{Q}{%
  \refstepcounter{numberingbase}\setcounter{QX}{\value{numberingbase}}%
  \begin{QX}}{\end{QX}}
\crefname{QX}{Question}{Questions}
\Crefname{QX}{Question}{Questions}

\newtheorem{variantX}{Variant}[subsection]

\crefname{variantX}{Variant}{Variants}
\Crefname{variantX}{Variant}{Variants}
\newtheorem{quesX}{Question}[subsection]

\newenvironment{ques}{%
  \refstepcounter{numberingbase}\setcounter{quesX}{\value{numberingbase}}%
  \begin{quesX}}{\end{quesX}}
\crefname{quesX}{Question}{Questions}
\Crefname{quesX}{Question}{Qeustions}

\newtheorem{goalX}{Goal}[subsection]

\crefname{goalX}{Goal}{Goals}
\Crefname{goalX}{Goal}{Goals}

\newtheorem{defpropX}{Definition/Proposition}[subsection]

\crefname{defpropX}{Definition/Proposition}{Definition/Propositions}
\Crefname{defpropX}{Definition/Proposition}{Definition/Propositions}

\theoremstyle{defs}

\newtheorem{defnX}{Definition}[subsection]

\newenvironment{defn}{%
  \refstepcounter{numberingbase}\setcounter{defnX}{\value{numberingbase}}%
  \begin{defnX}}{\end{defnX}}
\crefname{defnX}{Definition}{Definitions}
\Crefname{defnX}{Definition}{Definitions}

\newtheorem{definitionX}{Definition}[subsection]

\newenvironment{definition}{%
  \refstepcounter{numberingbase}\setcounter{definitionX}{\value{numberingbase}}%
  \begin{definitionX}}{\end{definitionX}}
\crefname{definitionX}{Definition}{Definitions}
\Crefname{definitionX}{Definition}{Definitions}

\newtheorem{ConstructionX}{Construction}[subsection]

\crefname{ConstructionX}{Construction}{Constructions}
\Crefname{ConstructionX}{Construction}{Constructions}
\newtheorem*{strategy*}{Strategy}

\newtheorem{propconsX}{Proposition/Construction}[subsection]

\crefname{propconsX}{Proposition/Construction}{Proposition/Constructions}
\Crefname{propconsX}{Proposition/Construction}{Proposition/Constructions}

\newtheorem{defconsX}{Definition/Construction}[subsection]

\crefname{defconsX}{Definition/Construction}{Definition/Constructions}
\Crefname{defconsX}{Definition/Construction}{Definition/Constructions}

\newtheorem{cautionX}{Caution}[subsection]

\crefname{cautionX}{Caution}{Cautions}
\Crefname{cautionX}{Caution}{Cautions}

\newtheorem{egX}{Example}[subsection]

\newenvironment{eg}{%
  \refstepcounter{numberingbase}\setcounter{egX}{\value{numberingbase}}%
  \begin{egX}}{\end{egX}}
\crefname{egX}{Example}{Examples}
\Crefname{egX}{Example}{Examples}

\newtheorem{remX}{Remark}[subsection]

\newenvironment{rem}{%
  \refstepcounter{numberingbase}\setcounter{remX}{\value{numberingbase}}%
  \begin{remX}}{\end{remX}}
\crefname{remX}{Remark}{Remarks}
\Crefname{remX}{Remark}{Remarks}

\newtheorem{RemarkX}{Remark}[subsection]

\crefname{RemarkX}{Remark}{Remarks}
\Crefname{RemarkX}{Remark}{Remarks}

\newtheorem{remarkX}{Remark}[subsection]

\crefname{remarkX}{Remark}{Remarks}
\Crefname{remarkX}{Remark}{Remarks}

\newtheorem*{egs}{Examples}
\newtheorem*{rems}{Remarks}

\theoremstyle{claims}
\newtheorem{claim}[equation]{Claim}

\Crefname{claim}{Claim}{Claims}
\Crefname{bclaim}{Claim}{Claims}
\Crefname{sublemma}{Lemma}{Lemmas}

\Crefname{thmenumi}{Theorem}{Theorems}
\AtBeginEnvironment{thmX}{%
  \crefalias{enumi}{thmenumi}%
  \setlist[enumerate,1]{label={\textit{(\arabic*)}},ref={\thethmX.(\arabic*)}}}

\Crefname{propenumi}{Proposition}{Propositions}
\AtBeginEnvironment{propX}{%
  \crefalias{enumi}{propenumi}%
  \setlist[enumerate,1]{label={\textit{(\arabic*)}},ref={\thepropX.(\arabic*)}}}

\newtheoremstyle{subsection-tweak}
   {2pt}
   {3pt}%
   {}
   {}%
   {\bfseries}
   {}%
   {.5em}
   {\thmnumber{\@{#1}{}\@{#2}.}%
    \thmnote{~{\bfseries#3.}}}    
    
\theoremstyle{subsection-tweak}
\newtheorem{pp}[numberingbase]{}
\newcommand{\bpp}{\begin{pp}}
\newcommand{\epp}{\end{pp}}

\theoremstyle{subsection-tweak}
\newtheorem{pp-tweak}[subsection]{}

\makeatletter
\def\@tocline#1#2#3#4#5#6#7{
    \begingroup 
    \@ifempty{#4}{}{}

    \parindent\z@ \leftskip#3\relax \advance\leftskip\@tempdima\relax
    #5\hskip-\@tempdima
      \ifcase #1
       \or\or \hskip 2em \or \hskip 1em \else \hskip 3em \fi%
      #6\nobreak\relax
    \dotfill\hbox to\@pnumwidth{\@tocpagenum{#7}}\par
    \nobreak
    \endgroup
 }
 \def\l@section{\@tocline{1}{0pt}{1pc}{}{}}

\renewcommand{\tocsection}[3]{%
  \indentlabel{\@ifnotempty{#2}{\makebox[1.3em][l]{%
    \ignorespaces#1 \bfseries{#2}.\hfill}}}\bfseries{#3}
    \vspace{-5pt}}

\renewcommand{\tocsubsection}[3]{%
  \indentlabel{\@ifnotempty{#2}{\hspace*{-0.5em}\makebox[2.1em][l]{%
    \ignorespaces#1#2.\hfill}}}#3
    \vspace{-5pt}}
   
\makeatother 

\makeatletter 
\newcommand\appendix@section[1]{%
  \refstepcounter{section}%
  \orig@section*{Appendix \@Alph\c@section. #1}%
}
\let\orig@section\section
\g@addto@macro\appendix{\let\section\appendix@section}
\makeatother
\makeatletter
\def\l@subsubsection{\@tocline{3}{1.8em}{3.2em}{\small}{}}
\makeatother

\title{A Note on Hodge theoretic anabelian geometry}
\author{qixiang wang}
\date{October 2025}

\begin{document}

\maketitle
\begin{abstract}
Grothendieck’s anabelian conjectures predict that certain classes of varieties over number fields are largely determined by their étale fundamental groups. A theorem of Mochizuki shows that for hyperbolic curves over number fields or $p$-adic fields, dominant morphisms bijectively correspond to open homomorphisms between their étale fundamental groups.

Motivated by non-abelian Hodge theory, we formulate a Hodge-theoretic version of anabelian conjecture in which the Galois action is replaced by the natural $\mathbb{C}^*$-action on the pro-algebraic completion of the fundamental group arising from non-abelian Hodge theory. In particular, we prove a Hodge theoretic analog of Mochizuki's theorem for smooth projective hyperbolic curves over $\mathbb{C}$. We also obtain a higher-dimensional analogue for complex hyperbolic manifolds of ball quotient type and discuss possible extensions to non-$K(\pi,1)$ spaces replacing fundamental groups by homotopy types.
\end{abstract}
\section{Introduction}
Over a number field $K$, the arithmetic geometry of hyperbolic curves (i.e., smooth proper curves of genus $\geq 2$) is extremely rich. One famous example is the Mordell conjecture (now Faltings' theorem), which asserts that hyperbolic curves have only finitely many rational points. Its proof, together with proofs of its geometric analogue (the Mordell conjecture for hyperbolic curves over function fields), suggests a deeper connection between arithmetic geometry and hyperbolic geometry.

One crucial feature of hyperbolic geometry is that hyperbolic manifolds (especially in higher dimensions) have complicated fundamental groups---sufficiently complicated to capture most of the geometric information. More precisely, the Mostow rigidity theorem states:

\begin{thm}[Mostow rigidity]
    Let $X$ and $Y$ be compact hyperbolic manifolds with $\dim X,\dim Y>2$. Then $\pi_1(X)\simeq \pi_1(Y)$ if and only if $X\simeq Y$ as Riemannian manifolds. Moreover,
    \[
      \mathrm{Isom}_{\mathrm{gp}}\bigl(\pi_1(X),\pi_1(Y)\bigr)\simeq \mathrm{Isom}(X,Y).
    \]
\end{thm}

As an arithmetic analogue, Grothendieck proposed the \emph{anabelian geometry program}, which roughly predicts that the category of varieties over $K$ with ``complicated'' étale fundamental groups should embed fully faithfully into the category of profinite groups via the functor taking étale fundamental groups.

A concrete result in this direction is the following theorem (also known as the \emph{Hom conjecture}) due to S.\ Mochizuki.

\begin{thm}[\cite{anabelian}]\label{main}
    Let $K$ be a $p$-adic field or a number field, and let $G_{K}$ be the Galois group of $K$. Let $X$ and $Y$ be hyperbolic curves over $K$. Then the natural map 
    \[
      \mathrm{Hom}^{\mathrm{dom}}(X,Y)\longrightarrow \mathrm{Hom}_{G_K}^{\mathrm{open}}\bigl(\pi_1(X),\pi_1(Y)\bigr)
    \]
    is bijective. Here, $\mathrm{Hom}^{\mathrm{dom}}(-,-)$ is the set of dominant morphisms and $\mathrm{Hom}_{G_{K}}^{\mathrm{open}}\bigl(\pi_1(X),\pi_1(Y)\bigr)$ denotes the set of open homomorphisms between fundamental groups that are compatible with the fiber sequences
    \[
      1\longrightarrow \pi_1(\bar{Z})\longrightarrow\pi_1(Z)\longrightarrow G_{K}\longrightarrow 1 \qquad (Z=X \text{ or } Y).
    \]
\end{thm}

Notice that the Galois group of the base field plays an essential role in the statement itself. Indeed, over an algebraically closed field (e.g., $\mathbb{C}$), the set of homomorphisms between fundamental groups is much coarser data than the set of morphisms between curves. In particular, the existence of an isomorphism between the fundamental groups of hyperbolic curves is equivalent to the equality of their genus.

Inspired by the work of Hitchin and Simpson on non-abelian Hodge theory, we formulate a complex-analytic analogue of \Cref{main} for complex hyperbolic Riemann surfaces in which the Galois group $G_K$ is replaced by the ``$\mathbb{C}^*$-action'' arising from non-abelian Hodge theory.

Non-abelian Hodge correspondence states that for smooth proper $X$ there is an equivalence of categories\upshape:
\[
\begin{aligned}
\Bigl\{
  \begin{aligned}
  &\textit{finite-dimensional $\mathbb{C}-$representations}\\
  &\textit{of the fundamental group }\pi_1(X)
  \end{aligned}
\Bigr\}
\;\simeq\;
\Bigl\{
  \begin{aligned}
  &\textit{semistable Higgs bundles on }X\\
  &\textit{with vanishing Chern classes}
  \end{aligned}
\Bigr\}
\end{aligned}
\]

It can be viewed as a non-abelian version of Hodge decomposition. On the Higgs side there is a natural $\mathbb{C}^{*}$-action given by rescaling the Higgs field, $\theta\mapsto t\theta$. This suggests that on the local-system side there should be a hidden ``$\mathbb{C}^{*}$-action'' on $\pi_1(X)$, and that $\mathbb{C}^{*}$ should be viewed as the ``motivic Galois group'' of the complex numbers. With this in mind, we propose and prove the following Hodge-theoretic version of \Cref{main}.

\begin{thm}[\Cref{anabelian for curve},\Cref{ball quotient anabelian}]\label{hodge anabelian}
    Let $X$ be a compact Kähler maniforld and $Y/\mathbb{C}$ a smooth projective hyperbolic curve. Then the natural map
    \[
      \mathrm{Hom}^{\mathrm{dom}}(X,Y)\simeq \mathrm{Hom}_{\mathbb{C}^{*}}^{\mathrm{open}}\bigl(\pi_1(X), \pi_1(Y)\bigr)
    \]
    is bijective. We explain the definition of $\mathrm{Hom}_{\mathbb{C}^*}^{\mathrm{open}}(-,-)$ in \Cref{open and equiv}.
\end{thm}

The proof is relatively simple, but relies crucially on non-abelian Hodge theory. Using a similar method, we also prove a higher-dimensional generalization.
\begin{thm}\label{higher}
     Let $X$ be a compact Kähler manifold and $Y$ be (compact) complex hyperbolic manifold of ball quotient type. Then the natural morphism 
    \[
      \pi\colon \mathrm{Hom}(X,Y)\rightarrow \mathrm{Hom}_{\mathbb{C}^*}\bigl(\pi_1(X),\pi_1(Y)\bigr)
    \]
induces a isomorphism $$\pi^{-1}(\mathrm{Hom}_{\mathbb{C}^*}^{\mathrm{open}}(\pi_1(X),\pi_1(Y)))\simeq \mathrm{Hom}_{\mathbb{C}^*}\bigl(\pi_1(X),\pi_1(Y)\bigr) $$
\end{thm}

Recent developments in anabelian geometry suggest that anabelian phenomenon also occurs for  certain non-$K(\pi,1)$ spaces, if one considers the entire étale homotopy type instead of étale fundamental group. In particular, the following theorem was proven in \cite{Schmidt_2016}.

\begin{thm}[\cite{Schmidt_2016}*{Theorem 1.2}]
    Let $K/\mathbb{Q}$ be a finitely generated extension of $\mathbb{Q}$. Let $X$ and $Y$ be smooth, geometrically connected varieties over $K$ which can be embedded as locally closed subschemes into a product of hyperbolic curves over $K$. Then
    \[
      \mathrm{Isom}_{K}(X,Y)\longrightarrow \mathrm{Isom}_{(\Spec K)_{\et}^{\sim}}\bigl((X_{\et})^{\sim},(Y_{\et})^{\sim}\bigr)
    \]
    admits a retraction. Here, $(-)_{\et}^{\sim}$ denotes the étale homotopy type of schemes.
\end{thm}

In the complex analytic context we choose the schematic homotopy type constructed in \cite{Katzarkov_2008} to be the analogue of étale homotopy type, we will formulate an analogous Hodge theoretical version of it. 
\begin{conj}
Let $X$ and $Y$ be smooth, geometrically connected varieties over $\mathbb{C}$ which can be embedded as closed subschemes into a product of hyperbolic curves over $\mathbb{C}$. Then
\[
  \mathrm{Isom}_{\mathbb{C}}(X,Y)\longrightarrow \mathrm{Isom}_{\mathbb{C}^*}\bigl(X(\mathbb{C})^{\sim},Y(\mathbb{C})^{\sim}\bigr)
\]
admits a retraction. We refer to \Cref{hodge homotopy isom} the definition of right hand side.
\end{conj}

The origin of this note lies in an attempt to understand \Cref{main} through the lens of \(p\)-adic non-abelian Hodge theory (often referred to as the \(p\)-adic Simpson correspondence; see, for example, \cite{heuer2025padicsimpsoncorrespondencesmooth}). Mochizuki’s proof relies crucially on \(p\)-adic Hodge theory, which naturally suggests that a more refined framework—namely \(p\)-adic non-abelian Hodge theory—might offer a new perspective. It is worth noting that Mochizuki’s argument is highly technical and difficult to adapt directly to our Hodge-theoretic setting, whereas our approach is comparatively simple and direct.

Motivated by the proof of \Cref{hodge anabelian}, we are led to believe that a deeper study of certain uniformizing local systems\footnote{Namely, the \(\mathbb{C}_p\)-local system whose associated Higgs bundle under the \(p\)-adic Simpson correspondence is the uniformizing Higgs bundle defined in \cref{uniformizing}. The existence of such local systems has recently been established conditionally in unpublished work of Xu--Zhu.} on hyperbolic curves over \(\mathbb{C}_p\) could yield a more conceptual proof of \Cref{main}. As these ideas are still preliminary and logically independent of the present work, we do not pursue them further here, and instead refer the interested reader to future work for additional details.
\begin{pp-tweak}[\!\!Acknowledgements]
The author would like to thank Kęstutis Česnavičius for his support. The author is also grateful to Daniel Litt, Vadim Vologodsky and Sasha Petrov for helpful discussions.
\end{pp-tweak}

\section{Non-abelian Hodge theory and the $\mathbb{C}^*$-action}
\subsection{Recollection of non-abelian Hodge theory}
For a compact Kähler manifold $(X,\omega)$, the non-abelian Hodge correspondence established by Simpson relates topological objects---local systems---to algebraic objects---Higgs bundles. Recall that a Higgs bundle $(E,\theta)$ on $X$ is a vector bundle $E$ on $X$ together with an $\cO_X$-linear morphism
\[
  \theta\colon E\longrightarrow E\otimes \Omega_X^{1}, \theta\wedge \theta=0.
\]
We call a Higgs bundle $(E,\theta)$ \textit{semistable} if every Higgs-sub-bundle $F\subset E$, we have $\mathrm{deg}(F)/\mathrm{rk(E)}\leq \mathrm{deg}(F)/\mathrm{rk}(F)$ where $\mathrm{deg}(F)\coloneq \int_{X}c_1(F)\wedge \omega^{n-1}$. We call a Higgs bundle $(E,\theta)$ has vanishing Chern class if $\int_{X}c_1(E)\wedge \omega^{n-1}=\int_Xc_2(E)\wedge \omega^{n-2}=0$.
The non-abelian Hodge correspondence is as follows.

\begin{thm}[\cite{PMIHES_1992__75__5_0}]\label{simpson nonab}
    There is an equivalence of categories:
    \[
\begin{aligned}\label{nonabhodge}
\Bigl\{
  \begin{aligned}
  &\textit{Finite-dimensional $\mathbb{C}-$representations}\\
  &\textit{of the fundamental group }\pi_1(X)
  \end{aligned}
\Bigr\}
\;\simeq\;
\Bigl\{
  \begin{aligned}
  &\textit{Semistable Higgs bundles on }X\\
  &\textit{with vanishing Chern classes}
  \end{aligned}
\Bigr\}
\end{aligned}
\]
  Moreover, for each reductive group $G$, this induces a homeomorphism between the moduli space of representations of the fundamental group and the moduli space of \up{semistable with vanishing Chern classes} Higgs bundles:
    \begin{equation}\label{homeo}
        \mathrm{Rep}_{G(\mathbb{C})}\bigl(\pi_1(X)\bigr)(\mathbb{C})\simeq \mathrm{Higgs}_{G}^{\mathrm{ss},\,\mathrm{Ch}=0}(X)(\mathbb{C}),
    \end{equation}
    which is functorial with respect to homomorphisms of reductive groups $G$.
\end{thm}

We recall a few basic facts about the non-abelian Hodge correspondence.

\begin{prop}[\cite{PMIHES_1992__75__5_0}]\label{action}
    In the setting above:
    \benuma
    \item Let $\cL$ be a local system on $X$. If the Higgs bundle $(E,\theta)$ corresponding to $\cL$ under \eqref{nonabhodge} is of the following form:
    \begin{equation}\label{graded}
        E\simeq \bigoplus_{i\in \mathbb{Z}}E_{i}, \qquad \theta=\bigoplus_{i\in\mathbb{Z}}\theta_i \text{ with } \theta_i\colon E_{i}\to E_{i-1}\otimes\Omega_X^1,
    \end{equation}
    where $E_i$ are vector bundles and $\theta_i$ are $\cO_X$-linear, then $\cL$ underlies a polarized variation of Hodge structure \up{(PVHS)}. We call a Higgs bundle of the form \eqref{graded} a \emph{graded Higgs bundle}.
    \item For any $t\in \mathbb{C}^*$, the scaling $\theta\mapsto t\theta$ defines a continuous $\mathbb{C}^*$-action on $\mathrm{Higgs}_G^{\mathrm{ss}}(X)(\mathbb{C})$. A Higgs bundle $(E,\theta)$ is fixed by this $\mathbb{C}^*$-action if and only if $(E,\theta)$ is a graded Higgs bundle.
    \item \up{Implied by (1)+(2).} Under \eqref{homeo}, the $\mathbb{C}^*$-action transfers to a continuous $\mathbb{C}^*$-action on the moduli space of representations of $\pi_1(X)$. In particular, a local system $\cL$ corresponds to a fixed point of this action if and only if it underlies a $\mathrm{PVHS}$.
    \eenum
\end{prop}

An important example of graded Higgs bundle and corresponding local system underlying a $\mathrm{PVHS}$ is the following.

\begin{eg}\label{uniformizing}
    Consider the following projective Higgs bundle $\mathbb{P}\bigl((E_X,\theta)\bigr)$\upshape:
    \[
      E_X\coloneq \cO_X \oplus \Omega_{X}^1,\qquad 
      \theta\colon \cO_X \oplus \Omega_{X}^1\to \Omega_{X}^1\oplus (\Omega_{X}^{1})^{\otimes 2}
      \ \text{ with }\ \theta|_{\cO_X}=0,\ \theta|_{\Omega_{X}^1}=(\mathrm{id},0),
    \]
    where $\mathbb{P}(E,\theta)$ denotes the associated projective bundle of $(E_X,\theta)$. We call $(E_X,\theta)$ the \emph{uniformizing Higgs bundle} \up{(since this construction does not depend on the dimension of $X$, we use the same term for any smooth algebraic variety)}.
    
    If $X$ is a hyperbolic curve, then this Higgs bundle, under non-abelian Hodge correspondence, corresponds to the $\mathrm{PSL}_2(\mathbb{R})$-representation of $\pi_1(X)$ defined by
    \[
      \rho_X\colon \pi_1(X)\longrightarrow \mathrm{Isom}(\mathbb{H})\simeq \mathrm{PSL}_2(\mathbb{R}),
    \]
    where the first map arises from the uniformization $\tilde{X}\simeq \mathbb{H}$. Up to  possibly choosing a double cover (an element in $H^1(X,\mathbb{Z}/2\mathbb{Z})$), this representation can be lifted to $\tilde{\rho}_X\colon \pi_1(X)\rightarrow \mathrm{SL}_2(\mathbb{R})$. The corresponding Higgs bundle, up to choosing a square root of $\Omega_X^1$, is the following 
     \[
      \tilde{E}_X\coloneq \Omega_X^{-\frac{1}{2}} \oplus \Omega_{X}^{\frac{1}{2}},\qquad 
      \theta\colon \Omega_X^{-\frac{1}{2}} \oplus \Omega_{X}^{\frac{1}{2}}\to \Omega_{X}^{\frac{1}{2}}\oplus \Omega_{X}^{\frac{3}{2}}
      \ \text{ with }\ \theta|_{\Omega_X^{-\frac{1}{2}}}=0,\ \theta|_{\Omega_{X}^{\frac{1}{2}}}=(\mathrm{id},0).
    \]
    This is a graded Higgs bundle (with $\Omega_X^{-\frac{1}{2}}$ weight $0$ part and $\Omega_{X}^{\frac{1}{2}}$ weight one part), thus $\tilde{\rho}_X$ give rise to a local system $\mathbb{L}_X$ on $X$ underlies a $\mathrm{PVHS}$-structure.

In fact, in \cite{Simpsonyangmill} Simpson reproved the uniformization theorem by considering the period map associated to $\mathbb{L}_X$: $p\colon \tilde{X}\rightarrow \mathbb{H}$, where $\tilde{X}$ is the universal cover. One observes that, the tangential map of $p$ can be computed by $\theta$, which in our case, is identity. Thus, $p$ is an isomorphism, since both are simply connected.

In the same spirit, Simpson provided the following example
\begin{eg}[\cite{Simpsonyangmill}*{Proposition 9.1, Proposition 9.8}]\label{simpson prop9.1}
    Let $X$ be a compact Kähler manifold of dimension $n$, consider the projective uniformizing Higgs bundle $\mathbb{P}(E_X,\theta)$. If we have that 
    $$(2c_2(X)-\frac{n}{n+1}c_1(X)^2)[\omega]^{n-2}=0$$
    Where $c_i(X)$ are the Chern classes and $[\omega]$ is the class of Khäler form, then $\mathbb{P}(E_X,\theta)$ gives rise to a stable projective Higgs bundle with vanishing Chern classes. Moreover, it give rise to a $\mathrm{PU}(n,1)\text{-}\mathrm{PVHS}$ on $X$, whose period map induces a isomorphism  of universal cover $\tilde{X}\simeq \mathbb{D}^n$ with the complex hyperbolic ball (period domain for $\mathrm{PU}(n,1)$). Conversely, if $X$ has universal cover by a complex hyperbolic ball, then it admits a $\mathrm{PU}(n,1)$-$\mathrm{PVHS}$ constructed as above.
 \end{eg}

\end{eg}

\subsection{$\mathbb{C}^*$-action on fundamental groups} 
Motivated by \Cref{action}(2), one expects a hidden $\mathbb{C}^*$-action on $\pi_1(X)$. Indeed, Simpson showed that, after choosing a base point $x\in X(\mathbb{C})$, there is a canonical $\mathbb{C}^{*}$-action (with $\mathbb{C}^*$ viewed as a discrete group) on $\pi_{1}(X,x)_{\mathbb{C}}^{\mathrm{alg}}$, the $\mathbb{C}$-linear pro-algebraic completion of $\pi_1(X,x)$ \cite[Theorem~6]{PMIHES_1992__75__5_0}. We now briefly explain a slight variant version.

Recall that for a abstract group $G$, the \textit{$\mathbb{C}$-linear pro\text{-}algebraic completion} of $G$ is the pro-algebraic group $G^{\mathrm{alg}}$ over $\mathbb{C}$ together with a map $G\rightarrow G^{\mathrm{alg}}(\mathbb{C})$ such that any finite dimensional $\mathbb{C}$-representation of $G$ factors through $G^{\mathrm{alg}}.$ Equivalently it is the Tannakian group of the Tannakian category of finite dimensional $\mathbb{C}$-representations of $G$ together with the fiber functor forgetting $G$-representation structure.
Without taking the fiber functor, this defines a gerbe $BG^{\mathrm{alg}}$ (see \cite{milne2007quotientstannakiancategories}).

Now, similarly, consider the Tannakian category $\mathrm{Loc}_{\mathbb{C}}(X)$ of $\mathbb{C}$-local systems on $X$, a choice of base point $x\in X$ gives a fiber functor $\mathrm{Loc}_{\mathbb{C}}(X)\rightarrow\mathrm{Vect}_{\mathbb{C}}$ by sending a local system to its stalk at $x$, which produces the Tannakian group $\pi_1(X,x)^{\mathrm{alg}}$. Without choosing base point, the Tannakian category then correspond to a gerbe $B\pi_1(X)^{\mathrm{alg}}$. 

By \Cref{simpson nonab}, the Tannakian category $\mathrm{Loc}_{\mathbb{C}}(X)$ is equivalent to the Tannakian category $\mathrm{Higgs}^{\mathrm{ss},\mathrm{Ch=0}}(X)$ consist of semistable Higgs bundles and vanishing Chern classes. Where on the later category there is $\mathbb{C}^{*}$ action as in \Cref{action}, which induces a $\mathbb{C}^*$ action on the gerbe $B\pi_1(X)^{\mathrm{alg}}$, or equivalently, a homomorphism
\begin{equation}\label{outer action}
    \mathbb{C}^*\rightarrow \pi_0\,\mathrm{Aut}_{\mathrm{Stack}/\mathbb{C}}\bigl(B \pi_1(X)^{\mathrm{alg}}\bigr).
\end{equation}

Concretely this means a $\mathbb{C}^*$-action on $\pi_1(X)^{\mathrm{alg}}$ up to conjugation:
\begin{equation}
    \mathbb{C}^*\rightarrow \mathrm{Out}\bigl(\pi_1(X)^{\mathrm{alg}}\bigr).
\end{equation}

Let $Y$ be another smooth projective variety over $\mathbb{C}$. By functoriality of pro-algebraic completion, the set $\mathrm{Hom}_{\mathrm{gp}}\bigl(\pi_1(X),\pi_1(Y)\bigr)/{\sim}$ of homomorphisms up to conjugation maps naturally to $\mathrm{Hom}_{\mathrm{alg\text{-}gp}}\bigl(\pi_1(X)^{\mathrm{alg}},\pi_1(Y)^{\mathrm{alg}}\bigr)/{\sim}$. More formally, we have
\[
  \mathrm{alg}(X,Y)\colon \pi_{0}\mathrm{Hom}_{\Ani}\bigl(B\pi_1(X),B\pi_1(Y)\bigr)\longrightarrow \pi_0 \mathrm{Hom}_{\mathrm{Stack}/\mathbb{C}}\bigl(B \pi_1(X)^{\mathrm{alg}},B\pi_1(Y)^{\mathrm{alg}}\bigr).
\]
By \eqref{outer action}, there is a $\mathbb{C}^*$-action on the target.

\begin{definition}\label{open and equiv}
In the above setting:
\benuma
\item A homomorphism $f\in \mathrm{Hom}_{\mathrm{gp}}\bigl(\pi_1(X),\pi_1(Y)\bigr)/\sim$ is \emph{$\mathbb{C}^{*}$-equivariant} if its image $\mathrm{alg}(X,Y)(f)$ is fixed under the $\mathbb{C}^*$-action.
\item A homomorphism $f\in \mathrm{Hom}_{\mathrm{gp}}\bigl(\pi_1(X),\pi_1(Y)\bigr)/\sim$ is \emph{open} if the image $f\bigl(\pi_1(X)\bigr)$ contains a finite-index subgroup of $\pi_1(Y)$.
\eenum
We denote by $\mathrm{Hom}_{\mathbb{C}^*}^{\mathrm{open}}\bigl(\pi_1(X),\pi_1(Y)\bigr)$ the set of homomorphisms that are both open and $\mathbb{C}^*$-equivariant.
\end{definition}

We now explain the relation between our notion of $\mathrm{Hom}_{\mathbb{C}^*}^{\mathrm{open}}(-,-)$ and the classical notion of $\mathrm{Hom}_{G_K}^{\mathrm{open}}(-,-)$ in \Cref{main}.

First, in the complex-analytic setting we do not possess a fully developed notion of an ``arithmetic fundamental group'' (or more precisely "Weil fundamental group") in the sense that we could complete an exact sequence of the form
\[
  1 \longrightarrow \pi_1(X) \longrightarrow ? \longrightarrow \pi_1^{\mathrm{Weil}}(\mathbb{C}) \coloneq \mathbb{C}^* \longrightarrow 1.
\]
In contrast, the definition of $\mathrm{Hom}_{G_K}^{\mathrm{open}}(-,-)$ involves homomorphisms between arithmetic fundamental groups on both the source and the target. However, a surjective homomorphism (up to conjugation) between $\pi_1^{\mathrm{\acute{e}t}}(X)$ and $\pi_1^{\mathrm{\acute{e}t}}(Y)$ over $G_K$ is equivalent (assuming $\pi_1^{\et}(\bar{Y})$ is center free) to a homomorphism (up to conjugation) between the geometric fundamental groups $\pi_1^{\mathrm{\acute{e}t}}(\bar{X})$ and $\pi_1^{\mathrm{\acute{e}t}}(\bar{Y})$ that is equivariant with respect to the outer action of $G_K$. Indeed, since in that case $\pi_1^{\et}(Y)$ can be reconstructed from $\pi_1^{\et}(\bar{Y})$ by the Cartesian diagram
on the left

\begin{minipage}{0.45\textwidth}
\small
\begin{tikzcd}[row sep=small, column sep=small]
\pi_1^{\et}(Y) \arrow[d] \arrow[r]                                  & G_K \arrow[d]                                                  \\
\mathrm{Aut}(\pi_1^{\et}(\bar{Y})) \arrow[r] \arrow[d, "-\circ f"]  & \mathrm{Out}(\pi_1^{\et}(\bar{Y})) \arrow[d, "-\circ f"]       \\
{\mathrm{Hom}(\pi_1^{\et}(\bar{X}),\pi_1^{\et}(\bar{Y}))} \arrow[r] & {\mathrm{Hom}(\pi_1^{\et}(\bar{X}),\pi_1^{\et}(\bar{Y}))/\sim}
\end{tikzcd}

\end{minipage}
\hfill
\begin{minipage}{0.45\textwidth}
\small
\begin{tikzcd}[row sep=small, column sep=small]
\pi_1^{\et}(X) \arrow[d] \arrow[r]                                  & G_K \arrow[d]                                                  \\
\mathrm{Aut}(\pi_1^{\et}(\bar{X})) \arrow[d, "f\circ -"]            & \mathrm{Out}(\pi_1^{\et}(\bar{Y})) \arrow[d, "-\circ f"]       \\
{\mathrm{Hom}(\pi_1^{\et}(\bar{X}),\pi_1^{\et}(\bar{Y}))} \arrow[r] & {\mathrm{Hom}(\pi_1^{\et}(\bar{X}),\pi_1^{\et}(\bar{Y}))/\sim}
\end{tikzcd}

\end{minipage}

where $f$ is the given $G_K$-equivariant surjection $\pi_1^{\et}(\bar{X})\rightarrow \pi_1^{\et}(\bar{Y})$. Then the commutative diagram on the right reconstruct the map $\tilde{f}\colon \pi_1^{\et}(X)\rightarrow \pi_1^{\et}(Y)$ from $f$.

From this point of view, our definition of an $\mathbb{C}^*$-equivariant homomorphism between fundamental groups is entirely analogous to the notion of a homomorphism between arithmetic fundamental groups over $G_K$.

Second, given a nontrivial open homomorphism of profinite groups $G_1 \to G_2$, the image necessarily contains a finite-index subgroup of $G_2$; indeed, such subgroups form a basis of the profinite topology on $G_2$.

In this sense, our definition of $\mathrm{Hom}_{\mathbb{C}^*}^{\mathrm{open}}$ is naturally analogous to $\mathrm{Hom}_{G_K}^{\mathrm{open}}$ in the classical anabelian geoemtry.

Finally, we record a basic property of $\mathbb{C}^*$-equivariant homomorphisms of fundamental groups.

\begin{prop}\label{pull back of pvhs}
In the above setting, let $\cL$ be a $\mathbb{C}$-local system on $Y$ that underlies a $\mathrm{PVHS}$. Then for any $f\in \mathrm{Hom}_{\mathbb{C}^*}^{\mathrm{open}}\bigl(\pi_1(X),\pi_1(Y)\bigr)$, the pullback local system $f^{*}\cL$ on $X$ also underlies a $\mathrm{PVHS}$.
\end{prop}

\begin{proof}
This is precisely \cite[Lemma 5.5 and Lemma 5.6]{PMIHES_1992__75__5_0}.
\end{proof}

\section{A Hodge-theoretic version of anabelian geometry}
\subsection{The case of hyperbolic curves}
We now formulate and prove the Hodge-theoretic Hom conjecture for hyperbolic curves.

\begin{thm}\label{anabelian for curve}
    Let $X, Y/\mathbb{C}$ be smooth projective hyperbolic curves. Then the natural map
    \[
      \mathrm{Hom}^{\mathrm{dom}}(X,Y)\simeq \mathrm{Hom}_{\mathbb{C}^{*}}^{\mathrm{open}}\bigl(\pi_1(X), \pi_1(Y)\bigr)
    \]
    is bijective. Here, $\mathrm{Hom}^{\mathrm{dom}}(X,Y)$ denotes the set of dominant morphisms $X\to Y$.
\end{thm}

\begin{proof}
\emph{Injectivity.} We present a simpler proof here although later the proof of \Cref{ball quotient anabelian} applies here as well. Let $F, G\colon X\to Y$ be dominant morphisms. If the induced maps on fundamental groups $\pi_1(F)$ and $\pi_1(G)$ are equivalent up to conjugation, we will show that $F=G$. Let $x_0\in X$ and set $y_0=G(x_0)$ and $y=F(x_0)$. Consider the induced maps on cohomology
\[
  F^*\colon H^1(Y,\mathbb{C})\to H^1(X,\mathbb{C}),\qquad G^*\colon H^1(Y,\mathbb{C})\to H^1(X,\mathbb{C}).
\]
The condition on fundamental groups implies $F^*=G^*$. By the functoriality of the Hodge decomposition and the Torelli theorem, this implies that the corresponding maps on the Jacobians $F_{*},G_{*}\colon Jac(X)\rightarrow Jac(Y)$ are equivalent.
Moreover, there is a diagram 
\[
\begin{tikzcd}
X \arrow[d, "AJ_{X,x}"'] \arrow[r, "G"', shift right=2] \arrow[r, "F", shift left] & Y \arrow[d, "AJ_{Y,y_0}"', shift right=2] \arrow[d, "AJ_{Y,y_0} + y", shift left] \\
\Jac(X) \arrow[r, "F_*=G_*"]                                                    & \Jac(Y)                                                               
\end{tikzcd}
\]
where the vertical maps are Abel--Jacobi maps, the inner and outer diagrams commute, respectively. Since the Abel--Jacobi maps embed $X$ and $Y$ injectively into their Jacobians, by the dominance of $F$ and $G$, this diagram shows that $Y$ have the same image under $AJ_Y$ and $AJ_Y + y$ in $Jac(T)$. But $Y$ and $Y+y$ intersect in $\Jac(Y)$ only at finitely many points, unless $F(x_0)=y=y_0=G(x_0)$. Thus, since our choice of $x_0$ is arbitrary, this shows $F=G$. 

\emph{Surjectivity.} Let $f\in \mathrm{Hom}_{\mathbb{C}^*}^{\mathrm{open}}\bigl(\pi_1(X),\pi_1(Y)\bigr)$. We construct a dominant morphism $\mathcal{F}\colon X\to Y$ such that the induced homomorphism on fundamental groups $\pi_1(\mathcal{F})$ is (up to conjugation) $f$.

Consider the $\mathrm{PSL}_2(\mathbb{R})$-local system $\cL_Y$ corresponding to $\rho_Y$ in \Cref{uniformizing}; this underlies a $\mathrm{PVHS}$. By \Cref{pull back of pvhs}, the local system on $X$ corresponding to $\rho_Y\circ f$ also underlies a $\mathrm{PVHS}$. This provides a period map on the universal cover $\tilde{X}$:
\[
\begin{tikzcd}
\tilde{X}\simeq\mathbb{H} \arrow[d] \arrow[r,"\tilde{F}"] & D=\mathrm{PSL}_2(\mathbb{R})/\mathrm{SO}(2)\simeq \mathbb{H} \\
X &                                            
\end{tikzcd}
\]
where $\mathbb{H}$ is the upper half-plane and $D$ is the period domain. By openness of $f$, the local system is not unitary, hence $\tilde{F}$ is nonconstant. Upon quotienting by the monodromy group, $\tilde{F}$ induces a holomorphic map
\[
  F'\colon X\longrightarrow D/\Gamma,
\]
where $\Gamma=\mathrm{Im}_{\rho_Y\circ f}\bigl(\pi_1(X)\bigr)$. Since $\rho_Y\bigl(\pi_1(Y)\bigr)$ is a faithful discrete subgroup of $\mathrm{PSL}_2(\mathbb{R})$ (hence acts properly discontinuously on $D$) and $\Gamma$ is a subgroup thereof, the quotient $D/\Gamma$ is a complex manifold. Composing with
\[
  G\colon D/\Gamma\longrightarrow D/\rho_{Y}\bigl(\pi_1(Y)\bigr)\simeq \mathbb{H}/\pi_1(Y)\simeq Y
\]
we obtain $F=G\circ F'\colon X\to Y$. By construction, the induced map on fundamental groups $\pi_1(F)$ is equivalent to $f$. This completes the proof.
\end{proof}

\subsection{Higher-dimensional anabelian geometry}
Instead of hyperbolic curves, a natural higher dimensional generalization is complex hyperbolic manifolds of ball quotient type. These are naturally $K(\pi,1)$ spaces, and of ``hyperbolic nature'' which makes them possible to be anabelian. We will confirm this speculation to some extend in this section.

Recall that the $n$ dimensional \textit{complex hyperbolic ball} is the domain
 $$\mathbb{D}^n\coloneq \{z\in \mathbb{C}^n| |z|^2<1\}.$$
It has holomorphic automorphism group $\mathrm{Aut}(\mathbb{D}^n)\simeq \mathrm{PU}(n,1)$ and carries the Bergman metric 
$$h=\frac{(1-|z|^2)\underset{i}{\Sigma}dz_i\tensor d\bar{z}_i+ (\underset{i}{\Sigma}\bar{z}_idz_i)\tensor(\underset{i}{\Sigma}z_id\bar{z}_i)}{(1-|z|^2)^2}$$
of constant sectional curvature $-1$. We call a compact Kähler manifold $B$ is of \textit{ball quotient type} if $B\simeq \mathbb{D}^n/\Gamma$ for some co-finite volume discrete subgroup $\Gamma\subset \mathrm{PU}(n,1)$. 
We will prove in this section that the ball quotient types are Hodge theoretic anabelian, more precisely, we have the following.
\begin{thm}\label{ball quotient anabelian}
    Let $X$ be a compact Kähler variety and $Y$ be complex hyperbolic manifolds of ball quotient type. Then the natural morphism 
    \[
      \pi\colon \mathrm{Hom}(X,Y)\rightarrow \mathrm{Hom}_{\mathbb{C}^*}\bigl(\pi_1(X),\pi_1(Y)\bigr)
    \]
induces aa isomorphism $$\pi^{-1}(\mathrm{Hom}_{\mathbb{C}^*}^{\mathrm{open}}(\pi_1(X),\pi_1(Y)))\simeq \mathrm{Hom}_{\mathbb{C}^*}\bigl(\pi_1(X),\pi_1(Y)\bigr) $$
Moreover, assume $X$ is also of ball quotient type, then it induces an isomorphism
 $$\mathrm{Isom}(X,Y)\simeq \mathrm{Isom}_{\mathbb{C}^*}(\pi_1(X),\pi_1(Y)).$$
\end{thm}

Note the last part yields a statement essentially weaker than Siu’s holomorphic Mostow rigidity \cite{Siumostow}. It is plausible that the strategy used in the proof of it could provide an alternative proof of Siu’s result.

\begin{rem}
    The formation of $\Cref{ball quotient anabelian}$ is a bit artificial. This is caused by the fact that it is unclear that given a holomorphic mapping $f\colon X\rightarrow Y$ between ball quotient types, the induced map on fundamental group $\pi_1(f)\colon \pi_1(X)\rightarrow \pi_1(Y)$ being surjective is suffices to show that $f$ is dominant. This will be true to some extend when $f$ is smooth by \cite{koziarz2008nonexistenceholomorphicsubmersionscomplex}.
\end{rem}
Before we proceed to the proof, let us briefly recall the notion of $(G_{\mathbb{R}},h)$-polarized variation of Hodge structures for a real reductive group $G_{\mathbb{R}}$ and a Hodge cocharacter $h\colon \mathbb{S}\simeq \mathrm{Res}_{\mathbb{C}/\mathbb{R}}\mathbb{G}_m\rightarrow G_{\mathbb{R}}$.
 
\begin{definition}Let $G$ be a real reductive group and $h\colon \mathrm{Res}_{\mathbb{C}/
    \mathbb{R}} \mathbb{G}_m\rightarrow G$ be a Hodge cocharacter. 
  \benuma 
  \item 
    A $G$-\textit{variation of Hodge structure} ($G$-$\mathrm{PVHS}$) on a complex manifold $X$ is a $G(\mathbb{R})$-local system $\mathbb{L}$ on $X$ such that for every faithful representation $\rho\colon G\rightarrow \mathrm{GL}(V_{\mathbb{R}})$, the induced local system $\mathbb{L}\times_{G_(\mathbb{R})}V_{\mathbb{R}}$ admits a polarised variation of Hodge structure with respect to the Hodge cocharacter $\rho\circ h$.
    \item Let $H_{h}$ be the subgroup of $G(\mathbb{R})$ that stabilize the Hodge cocharacter $h$. Then the real analytic manifold $D_{G,h}\coloneq G(\mathbb{R})/H_h$ admits a complex manifold structure, and we call it the \textit{Period domain} associated to $(G,h).$ 
    \item  Recall that by \cite{milne2011shimuravarietiesmoduli}*{Theorem 7.3} that for a complex manifold $X$ and its universal cover $\tilde{X}$, a $G$-$\mathrm{PVHS}$ $\mathbb{L}$ on $X$ induces a holomorphic map $f_{\mathbb{L}}\colon \tilde{X}\rightarrow D_{G,h}$ that is equivariant with respect to $\pi_1(X)$ (here, $\pi_1(X)$ acts on $D_{G,h}$ via composing the action of $G(\mathbb{R})$ on $D_{G,h}$ with the underlying representation $\rho_{\mathbb{L}}\colon \pi_1(X)\rightarrow G(\mathbb{R})$). We call this map $f_{\mathbb{L}}$ the \textit{period map} associated to $\mathbb{L}.$
     \eenum
     Composing $h$ with the conjugate action $G\rightarrow \mathrm{Aut}(G)$ we obtain a $\mathbb{C}^*$-action on $G$. We call this the \textit{$\mathbb{C}^*$-action on $G$ associated to $h$}.
\end{definition}
By Tannakian formalism, a $G$-local system $\mathbb{L}$ on a complex Kähler manifold $X$ is equivalent to a morphism of (pro) algebraic groups $\pi_{\mathbb{L}}\colon\pi_1^{\mathrm{alg}}(X)\rightarrow G$ up to conjugation. 
\begin{prop}\label{c star equiv map}
    Let $G$ be a real reductive group and $h\colon \mathrm{Res}_{\mathbb{C}/
    \mathbb{R}} \mathbb{G}_m\rightarrow G$ a Hodge cocharacter. If $\mathbb{L}$ is a local system on a compact Kähler manifold $X$ that corresponds to a morphism of pro-algebraic groups $\pi_{\mathbb{L}}\colon \pi_1^{\mathrm{alg}}(X)\rightarrow G$, then $\mathbb{L}$ admits a $G$-$\mathrm{PVHS}$ structure if and only if that $\pi_{\mathbb{L}}$ is equivariant with respect to the $\mathbb{C}^*$ action. 

    Moreover, if $\mathbb{L}$ is irreducible and $\pi_{\mathbb{L}}$ is $\mathbb{C}^*$ equivariant, then $\mathbb{L}$ admits an unique $G$-$\mathrm{PVHS}$ structure. 
\end{prop}
\begin{proof}
    This is exactly \cite{PMIHES_1992__75__5_0}*{Lemma 5.8}. 
\end{proof}

\begin{proof}[Proof of \Cref{ball quotient anabelian}]
We first show the surjectivity. That is for any $\mathbb{C}^*$ equivariant open morphism $f\colon \pi_1(X)\rightarrow \pi_1(Y)$, we will construct a map $F\colon X\rightarrow Y$ such that $\pi_1(F)= f$ up to conjugation. Notice that by passing to finite étale cover, one may assume that $f$ is surjective. Consider the representation $\rho_Y\colon \pi_1(Y)\rightarrow \mathrm{PU}(n,1)\simeq \mathrm{Aut}(\mathbb{D}^n)$ induced from the presentation $Y\simeq \mathbb{D}^n/ \pi_1(Y).$ By \Cref{simpson prop9.1}, the $\mathrm{PU}(n,1)$ local system $\mathbb{L}_Y$ corresponds to $\rho_Y$ underlies a $\mathrm{PU}(n,1)$-$\mathrm{PVHS}$ structure, moreover its monodromy group is Zariski dense. By \Cref{c star equiv map}, we see that this gives rise to a $\mathbb{C}^*$-equivariant surjection $\rho_Y^{\mathrm{alg}}\colon \pi_1^\mathrm{alg}(Y)\surjects \mathrm{PU}(n,1)$. 

Since $f$ is $\mathbb{C}^*$ equivariant, we have that $\rho_Y^{\mathrm{alg}}\circ f^{\mathrm{alg}}$ gives a $\mathbb{C}^*$ equivariant surjection $\pi_1^{\mathrm{alg}}(X)\surjects \mathrm{PU}(n,1)$. Thus we have that $\rho_Y\circ f$ corresponds to a $\mathrm{PU}(n,1)$-$\mathrm{PVHS}$ on $X$. Consider the period map of it gives the following holomorphic map\textup{:}
$$X\rightarrow \mathbb{D}^n/\rho_Y\circ f(\pi_1(X)).$$
Further compose with the natural projection $\mathbb{D}^n/\rho_Y\circ f(\pi_1(X))\rightarrow \mathbb{D}^n/\rho_Y(\pi_1(Y))\simeq Y$ gives the desired morphism $F\colon X\rightarrow Y.$ Notice that such a construction is canonical since, by irreducibility of $\rho_Y$ and surjectivity of $f$, $\rho_Y\circ f$ is also irreducible, thus by \Cref{c star equiv map} there is a unique $\mathrm{PU}(n,1)\text{-}\mathrm{PVHS}$ structure associate to $\rho_Y\circ f$.

Now we show the injectivity. Let $F\colon X\rightarrow Y$ be a morphism such that the induced map $\pi_1(F)\colon \pi_1(X)\rightarrow \pi_1(Y)$ is open. Since the image is of finite index, $G$ in fact factors through a finite étale cover $Y'$ of $Y$ such that the induce map on fundamental groups $\pi_1(X)\rightarrow \pi_1(Y')$ is surjective. This reduction allows us to assume $\pi_1(F)$ is surjective. Let $\tilde{F}$ be the map constructed above from $\pi_1(G)$, since such $\tilde{F}$ uniquely determined by $\pi_1(F)$, it is sufficient to show that $F=\tilde{F}$. For this, suffices to show that the following diagram commutes
$$\begin{tikzcd}
\tilde{X} \arrow[r, "P_{F^*{\mathbb{L}_Y}}"] \arrow[d, "\pi_X"] & \mathbb{D}^n \arrow[d] \\
X \arrow[r, "F"']                                             & Y                       
\end{tikzcd}.$$
Here $P_{F^*\mathbb{L}_Y}$ is the period map associated to the $\mathrm{PU}(n,1)\text{-}\mathrm{PVHS}$ $F^*\mathbb{L}_Y$ and $\pi_X$ is the natural projection from the universal cover $\tilde{X}$ of $X$. In fact we shall expand the diagram to clarify the map $\mathbb{D}^n\rightarrow Y$.
$$\begin{tikzcd}
                                                                                                                                     &  & \mathbb{D}^n                                             \\
\tilde{X} \arrow[d, "\pi_X"'] \arrow[rru, "P_{F^*\mathbb{L}_Y}"] \arrow[rr, "(P_{\mathbb{L}_Y})^{-1}\circ P_{F^*{\mathbb{L}_{Y}}}"'] &  & \tilde{Y} \arrow[d, "\pi_Y"] \arrow[u, "P_{\mathbb{L}_Y}"] \\
X \arrow[rr, "F"']                                                                                                                   &  & Y                                                         
\end{tikzcd}$$
Where $P_{\mathbb{L}_Y}$ is the period map associated to $\mathbb{L}$ (which by proof of \cite{Simpsonyangmill}*{Proposition 9.1} is an isomorphism). Since all the maps are holomorphic, it suffices to prove that the diagram commutes on the set theoretic level. That is for any point $\tilde{x}\in \tilde{X}$, $y\coloneq F(\pi_X(\tilde{x}))= \pi_Y\circ (P_{\mathbb{L}_Y})^{-1}\circ P_{F^*{\mathbb{L}_{Y}}} (\tilde{x}).$ Or equivalently, choosing any lift $\tilde{y}\in \tilde{Y}$ of $y$, we need to show that up to $\pi_1(X)$ action, $P_{F^*\mathbb{L}_Y}(\tilde{x})= P_{\mathbb{L}_Y}(\tilde{y}).$ 

Embedding $\mathbb{D}^n$ into the flag variety $\mathbb{P}^n$, suffices to show that the flag on $\pi_X^*(F^*\mathbb{L}_Y\tensor_{\mathbb{R}} \cO_{X})$ restrict to $\tilde{x}$ is, up to $\pi_1(X)$ action, the same as the flag on $\pi_Y^* (\mathbb{L}_Y\tensor_{\mathbb{R}}\cO_Y) $ restricting to $\tilde{y}$. But by functoriality, the flag on $\tilde{x}$ is the same as the flag associated to the $\mathrm{PU}(n,1)\text{-}\mathrm{PHS}$ $\mathbb{L}_Y|_y$. Similarly by functoriality, the flag on $\pi_Y^* (\mathbb{L}_Y\tensor_{\mathbb{R}}\cO_Y) $ restricting to $\tilde{y}$ is, up to the $\pi_1(Y)$ action on trivialization, is the same as the flag on $\mathrm{PU}(n,1)\text{-}\mathrm{PHS}$ $\mathbb{L}_Y|_y$. This finishes the proof of the first part. 

For the second part, notice that both $X$ and $Y$ are $K(\pi_,1)$ spaces, thus by Whitehead theorem a map $f\colon X\rightarrow Y$ that induces isomorphism on fundamental groups induces an isomorphism on cohomology, in particular, an isomorphism $\mathrm{H}^{2n}(f)\colon \mathbb{Z}\simeq\mathrm {H}^{2n}(Y,\mathbb{Z})\rightarrow\mathrm{H}^{2n}(X,\mathbb{Z})\simeq \mathbb{Z}$. Thus the degree of $f$ is $1$, which implies that $f$ is birational. Therefore we have that $$K_X\sim f^*K_Y+ \underset{i}{\Sigma}a_i E_i$$ for $a_i\geq 0$ and $E_i$ are exceptional divisors. But now for any contracting curve $C\subset X$ (i.e., $f(C)=\mathrm{pt}$), we have that 
$$0<K_X\cdot C=f^*K_Y\cdot C + =(\underset{i}{\Sigma}a_i E_i)\cdot C\leq 0$$
which contradicts unless $f$ is isomorphism.

\end{proof}

\subsection{Anabelian phenomenon for homotopy type}
We end with some speculation about higher-dimensional Hodge-theoretic anabelian geometry for non-$K(\pi,1)$ spaces.

In \cite{Schmidt_2016}, the following higher-dimensional version of the Hom conjecture was proved.

\begin{thm}
    Let $K/\mathbb{Q}$ be a finitely generated extension of $\mathbb{Q}$. Let $X$ and $Y$ be smooth, geometrically connected varieties over $K$ which can be embedded as locally closed subschemes into a product of hyperbolic curves over $K$. Then
    \[
      \mathrm{Isom}_{K}(X,Y)\longrightarrow \mathrm{Isom}_{(\Spec K)_{\et}^{\sim}}\bigl((X_{\et})^{\sim},(Y_{\et})^{\sim}\bigr)
    \]
    admits a retraction. Here, $(-)_{\et}^{\sim}$ denotes the étale homotopy type of schemes.
\end{thm}

In the Hodge-theoretic context, we have a natural candidate for this ``arithmetic étale homotopy type.''

Recall that for a smooth proper scheme $S$ over $\mathbb{C}$, by \cite{Katzarkov_2008} one can functorially associate to the homotopy type $S(\mathbb{C})^{\sim}$ a $\infty\text{-}$stack $(S(\mathbb{C})^{\sim})^{\mathrm{sh}}\otimes \mathbb{C}$. In \cite{Katzarkov_2008}, it was shown that $(S(\mathbb{C})^{\sim})^{\mathrm{sh}}\otimes \mathbb{C}$ admits a $\mathbb{C}^{*}$-action (with $\mathbb{C}^*$ discrete). In particular, the cohomology of the stack $(S(\mathbb{C})^{\sim})^{\mathrm{sh}}\otimes \mathbb{C}$ is isomorphic to $H^*(S(\mathbb{C})^{\sim},\mathbb{C})$ the complex coefficient Betti cohomology of $S$,  where the $\mathbb{C}^*$ action realizing the Hodge decomposition. Moreover, by \cite{Katzarkov_2008}*{Proposition 1.1.8} there is a canonical isomorphism 
$$\tau_{\leq 1}\footnote{For any $\infty$-stack $X\in \mathrm{Shv}_{\mathrm{fpqc}}(\mathrm{Aff}^{\mathrm{op}},\cS)$ and $i\in \mathbb{N}$, $\tau_{\leq i}(X)$ is the stackification of the prestack sending any affine scheme $\mathrm{Spec}R$ to the $i$-truncated space $\tau_{\leq i}X(R)$.}((S(\mathbb{C})^{\sim})^{\mathrm{sh}}\otimes \mathbb{C})\simeq B \pi_1^{\mathrm{alg}}(X).$$ 
And the $\mathbb{C}^*$ action precisely gives Simpson’s construction in \eqref{outer action}.
\begin{definition}\label{hodge homotopy isom}
    Let $X$, $Y$ be smooth projective schemes over $\mathbb{C}$. We define the \emph{space of $\mathbb{C}^*$-equivariant isomorphisms of homotopy types} $\mathrm{Isom}_{\mathbb{C}^*}\bigl(X(\mathbb{C})^{\sim},Y(\mathbb{C})^{\sim}\bigr)$ by the Cartesian diagram in spaces
    \[
    \begin{tikzcd}
    {\mathrm{Isom}_{\mathbb{C}^*}\bigl(X(\mathbb{C})^{\sim},Y(\mathbb{C})^{\sim}\bigr)} \arrow[d] \arrow[r] & {\mathrm{Isom}_{\mathrm{Stacks}/\mathbb{C}}\bigl((X(\mathbb{C})^{\sim})^{\mathrm{sh}}\otimes \mathbb{C},(Y(\mathbb{C})^{\sim})^{\mathrm{sh}}\otimes \mathbb{C}\bigr)^{\mathbb{C}^{*}}} \arrow[d] \\
    {\mathrm{Isom}\bigl(X(\mathbb{C})^{\sim},Y(\mathbb{C})^{\sim}\bigr)} \arrow[r]                 & {\mathrm{Isom}_{\mathrm{Stacks}/\mathbb{C}}\bigl((X(\mathbb{C})^{\sim})^{\mathrm{sh}}\otimes \mathbb{C},(Y(\mathbb{C})^{\sim})^{\mathrm{sh}}\otimes \mathbb{C}\bigr)}                           
    \end{tikzcd}
    \]
    where in the upper-right corner, $(-)^{\mathbb{C}^*}$ denotes the homotopy fixed points under $\mathbb{C}^*$. 
\end{definition}
We propose the following conjecture \footnote{We thanks Emanuel Reinecke for the suggestion.}:
\begin{conj}
Let $X$ and $Y$ be smooth, geometrically connected varieties over $\mathbb{C}$ which can be embedded as closed subschemes into a product of hyperbolic curves over $\mathbb{C}$. Then
\[
  \mathrm{Isom}_{\mathbb{C}}(X,Y)\longrightarrow \mathrm{Isom}_{\mathbb{C}^*}\bigl(X(\mathbb{C})^{\sim},Y(\mathbb{C})^{\sim}\bigr)
\]
admits a retraction.
\end{conj}
     \bibliographystyle{plain}
\bibliography{Ref.bib}
\end{document}